\documentclass[oneside]{amsart}
\usepackage{amsmath, amssymb,color,amsthm,graphicx,cite,bm}
\newtheorem{theorem}{Theorem}[section]
\newtheorem{proposition}[theorem]{Proposition}
\newtheorem{lemma}[theorem]{Lemma}

\newtheorem{question}{Question}

\theoremstyle{definition}

\newtheorem{main}{Theorem}

\newtheorem{main_cor}[main]{Corollary}

\def\Z{\mathbb{Z} }
\def\R{\mathbb{R} }
\def\A{\mathbb{A} }

\def\nbd{neighborhood }
\def\nbds{neighborhoods }
\def\Sv{\mathop{\mathrm{Sing}}(v)}

\def\Pv{\mathop{\mathrm{Per}}(v)}
\def\Cv{\mathop{\mathrm{Cl}}(v)}

\author{Tomoo Yokoyama}
\date{\today}
\address{Applied Mathematics and Physics Division, Gifu University, Yanagido 1-1, Gifu, 501-1193, Japan\\}
\email{tomoo@gifu-u.ac.jp}
\thanks{The author was partially supported by JSPS Grant Number 20K03583 and 21H00980}
\subjclass[2010]{}
\title[Hamiltonian flows on unbounded surfaces]{Topological characterizations of Hamiltonian flows on unbounded surfaces}
\keywords{Hamiltonian flows; Unbounded surfaces}
\subjclass[2010]{Primary 37E35; Secondary 37J46,58K45}

\begin{document}
\maketitle

\begin{abstract}
Hamiltonian flows on compact surfaces are characterized, and the topological invariants of such flows with finitely many singular points are constructed from the viewpoints of integrable systems, fluid mechanics, and dynamical systems. Though various fluid phenomena are modeled as flows on the plane, it is not obvious to determine if the flows are Hamiltonian, even the singular point set is totally disconnected and every orbit is contained in a straight line parallel to the $x$-axis. In fact, there are such non-Hamiltonian flows on the plane. On the other hand, this paper topologically characterizes Hamiltonian flows on unbounded surfaces and constructs their complete invariant under a regularity condition for singular points. In addition, under finite volume assumption, Hamiltonian flows on unbounded surfaces can be embedded into those on compact surfaces. 
\end{abstract}

\section{Introduction}\label{intro}

The topological complete invariants of Hamiltonian flows with finitely many singular points on compact surfaces are constructed from integrable system points of views \cite{bolsinov1994classification,bolsinov1999exact,kruglikov1997exact,Oshemkov10} and fluid mechanics points of views \cite{aref1998stagnation,kidambi2000streamline,moffatt2001topology}. 
Moreover, the topological complete invariant is generalized into one of the Hamiltonian flows with finitely sectored  singular points without elliptic sectors on non-compact surfaces \cite{Nikolaenko20}. 
To analyze various fluid phenomena, the studies need be generalized into those for non-compact surfaces because uniform flows on planes have elliptic sectors at the ends. 
From dynamical system points of view \cite{ma2005geometric,sakajo2018tree,yokoyama2013word,yokoyama2021complete,yokoyamacot},
the structural stability of Hamiltonian flows on compact surfaces and non-compact punctured spheres are characterized, and their topological complete invariants are constructed. 
In addition, the generic intermediate flows between structurally stable Hamiltonian flows on compact surfaces and non-compact punctured spheres are also characterized, and their topological complete invariants are constructed \cite{sakajo2015transitions}. 
The higher combinatorial structures of the space of Hamiltonian vector fields with finitely many singular points on compact surfaces are described \cite{yokoyama2021combinatorial}. 
Furthermore, Hamiltonian flows with finitely many singular points on compact surfaces are topologically characterized using the non-wandering property \cite{yokoyama2021ham}. 
In fact, the difference and equivalence among Hamiltonian, area-preserving, and non-wandering properties for flows with finitely many singular points on compact surfaces are characterized topologically. 
On the other hand, 
``adding operations of totally disconnected singular points'' can break the Hamiltonian property for flows \cite{yokoyama2022omega}. 
In fact, there is a non-Hamiltonian flow on the plane generated by a vector field $fX$ with totally disconnected singular points, where $X =(1,0)$ is a Hamiltonian vector field and $f \colon \R^2 \to \R$ is a smooth function with totally disconnected zeros (see Proposition~\ref{prop:ex}). 

In this paper, Hamiltonian flows on unbounded surfaces are characterized topologically. 
In fact, under a regularity of singular points, we show that any flows on orientable surfaces with finitely many boundary components, finite genus, and finite ends are Hamiltonian if and only if they are flows without limit circuits or non-closed recurrent points such that the extended orbit spaces are finite directed graphs without directed cycles. 
Furthermore, the directed surface graph which is the finite union of centers, multi-saddles, and virtually border separatrices of such a Hamiltonian flow is a topological complete invariant. 
On the other hand, under finite volume assumption, any Hamiltonian flows with finitely many singular points on an orientable surface with finitely many boundary components, finite genus, and finite ends can be embedded in Hamiltonian flows on compact surfaces. 

\subsection{Statements of main results}
To state more precisely, we define some concepts. 
A {\bf parabolic sector} is topologically equivalent to a flow box with the point $(\pm \infty, 0)$ and a {\bf hyperbolic} (resp. {\bf elliptic}) {\bf sector} is topologically equivalent to a Reeb component (resp. the interior of a Reeb component) with the point $(\infty, 0)$ (resp. $(- \infty, 0)$) as in Figure~\ref{sectors_all}. 
\begin{figure}[t]
\begin{center}
\includegraphics[scale=0.15]{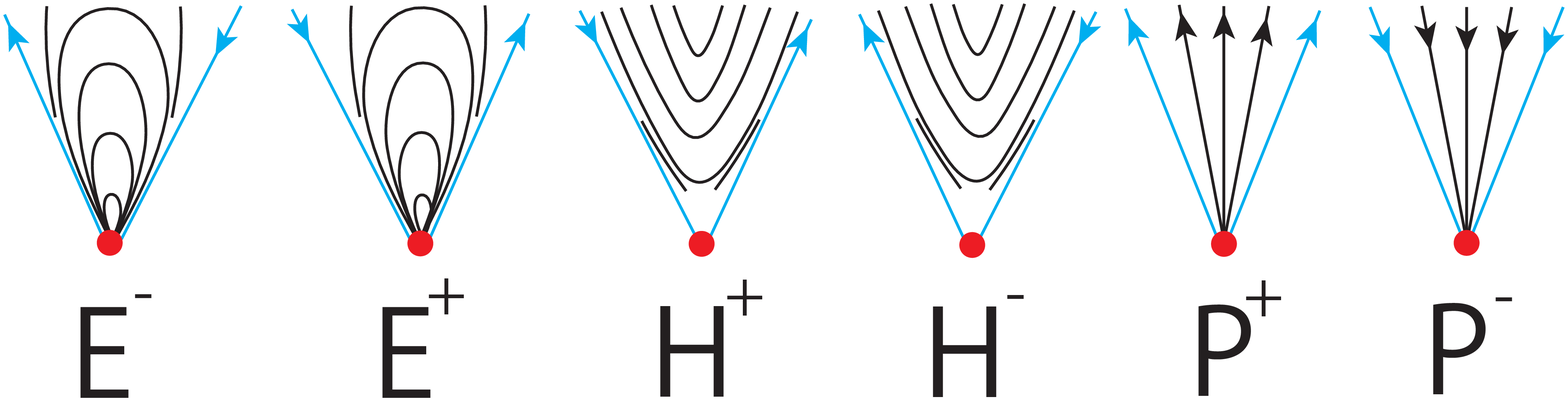}
\end{center}
\caption{Two parabolic sectors $P^-$ and $P^+$, two hyperbolic sectors $H^-$ and $H^+$ with clockwise and anti-clockwise orbit directions, and two elliptic sectors $E^+$ and $E^-$ with clockwise and anti-clockwise orbit directions respectively.}
\label{sectors_all}
\end{figure}
A singular point $x$ is {\bf finitely sectored} if either a center or there is an open neighborhood of $x$ which is an open disk and is a finite union of $\{ x \}$, parabolic sectors, hyperbolic sectors, and elliptic sectors such that each pair of distinct sectors intersects at most two orbit arcs. 
Note that a flow with finitely sectored singular points has at most finitely many singular points. 
A finitely sectored singular point is a {\bf multi-saddle} if any sectors are hyperbolic. 
The union of multi-saddles and separatrices from or to multi-saddles is called the {\bf multi-saddle connection diagram}. 
Any connected components of the multi-saddle connection diagram are called multi-saddle connections. 
A flow is of {\bf weakly finite type} if any singular points are finitely sectored, there are no non-closed recurrent points, and there are at most finitely many limit cycles.
We call that a flow $v$ on a surface $S$ with finitely many genus and ends {\bf has sectored ends} if any points in the set difference $S_{\mathrm{end}} - S$ are finitely sectored singular points on the end completion  $S_{\mathrm{end}}$. 

A non-singular point $x$ is a {\bf non-Hausdorff point} if there is a non-singular point $y \in S$ with $y \notin \overline{O(x)}$ and $x \notin \overline{O(y)}$ such that there are no invariant disjoint \nbds $U_x$ and $U_y$ of $x$ and $y$ respectively. 
Then the pair $x$ and $y$ is called a {\bf non-Hausdorff pair}. 
Define an equivalence relation $\sim_{\mathrm{nH}}$ on $S$ as follows: 
$x \sim_{\mathrm{nH}} y$ if either $O(x) = O(y)$ or the pair of $x$ and $y$ is a non-Hausdorff pair. 
For any non-Hausdorff point $x$, denote by $[x]_{\mathrm{nH}}$ the equivalent class of $x$.
For a flow $v$ of weakly finite type on a surface $S$ with finitely many boundary components, finite genus, and finite ends, the {\bf extended orbit} $O_{\mathrm{ex}}(x)$ of any non-Hausdorff point $x$ is defined as follows: 
\[
O_{\mathrm{ex}}(x) := [x]_{\mathrm{nH}} \cup \left( \bigcup_{y \in [x]_{\mathrm{nH}}} \alpha(y) \cup \omega(y) \right) \subset S
\]
For any point in the complement of the union of extended orbits of non-Hausdorff points, its extended orbit is defined as its orbit. 
Define an equivalence relation $\sim_{\mathrm{ex}}$ on $S$ as follows: 
$x \sim_{\mathrm{ex}} y$ if there is an extended orbit that contains $x$ and $y$. 
Then the quotient space $S/\sim_{\mathrm{ex}}$ is called the {\bf extended orbit space} and is denoted by $\bm{S/v_{\mathrm{ex}}}$. 
We have the following topological characterization of a Hamiltonian flow on a surface of ``compact type''. 

\begin{main}\label{main:01}
The following are equivalent for a flow $v$ with sectored ends and finitely many singular points on an orientable surface $S$ with finitely many boundary components, finite genus, and finite ends: 
\\
{\rm(1)} The flow $v$ is Hamiltonian. 
\\
{\rm(2)} The flow $v$ is a flow of weakly finite type and the extended orbit space $S/v_{\mathrm{ex}}$ is a finite directed graph without directed cycles.  

In any case, the following statement holds: 
\\
{\rm(a)} Each non-Hausdorff point in $S/v$ is either a multi-saddle or a virtually border separatrix. 
\\
{\rm(b)} The extended orbit space is the quotient space of the finite Reeb graph of its Hamiltonian by identifying non-Hausdorff points. 
\\
{\rm(c)} We can choose the graph structure of $S/v_{\mathrm{ex}}$ each of whose vertices is either a center, a multi-saddle connection, or a finite union of virtually border separatrices, and each of whose edges is either an invariant periodic annulus or an invariant trivial flow box. 
\end{main}

Roughly speaking, a virtually border separatrix is a boundary component of ``maximal invariant open periodic annulus" or ``maximal invariant open trivial flow box" (see definition in \S 2 for details).
The topological characterization of a Hamiltonian flow with sectored ends on an orientable surface of ``compact type'' is described in \S 3 (see Theorem~\ref{thm:02} for details). 
To demonstrate the following characterization, we consider the following question. 
\begin{question}
What kinds of isolated singular points do appear in Hamiltonian flows on {\rm(}possibly non-compact{\rm)} surfaces?
\end{question}

For the same question for gradient flows on surfaces (resp. non-wandering flows on compact surfaces), every isolated singular point of such a flow is a finitely sectored singular point without elliptic sectors which is not a center \cite[Theorem~A]{kibkalo2021topological} (resp. either a center or a multi-saddle \cite[Theorem 3]{cobo2010flows}).
We observe a characterization of isolated singular points for Hamiltonian flows on surfaces (see Lemma~\ref{lem:char_isolated_sing} for details). 

A {\bf directed surface graph} is an embedding of a finite directed multi-graph into a surface. 
The previous theorem implies the topological complete invariance of the finite union of centers, multi-saddles, and virtually border separatrices as follows. 

\begin{main_cor}\label{main_cor:01}
The directed surface graph which is the finite union of centers, multi-saddles, and virtually border separatrices of a Hamiltonian flow with sectored ends and finitely many singular points on an orientable surface with finitely many boundary components, finite genus, and finite ends is a topological complete invariant. 
\end{main_cor}

Though any Hamiltonian flows on compact surfaces are non-wandering, notice that a Hamiltonian flow on non-compact surfaces need not have non-wandering points. 
Indeed, a flow generated by a vector field $X = (1,0)$ on $\R^2$ is a Hamiltonian, but and any points are wandering. 
Thus Hamiltonian flows on compact surfaces and those on non-compact surfaces are quite different from each other. 
On the other hand, if Hamiltonian flows on non-compact surfaces preserve finite volumes, then they are non-wandering. 
Therefore we consider how different Hamiltonian flows on unbounded surfaces under finite volume conditions and not are, which is suggested by M. Shishikura. 
Under finite volume conditions, we show that such Hamiltonian flows are reduced to those on compact surfaces. 
More precisely, for a Hamiltonian vector field $X$ defined by $dH= \omega(X, \cdot )$ on a surface, the vector field $X$ is a Hamiltonian vector field on the surface with {\bf finite volume} if it the volume of the surface with respect to the volume form $\omega$ is finite.  
A flow is a Hamiltonian flow on a surface with {\bf finite volume} if it is topological equivalent to a flow generated by a Hamiltonian vector field on a surface with finite volume. 
We have the following characterization of a Hamiltonian vector field on a surface with finite volume.

\begin{main}\label{main:02}
The following are equivalent for a flow $v$ with finitely many singular points on an orientable surface $S$ with finitely many boundary components, finite genus, and finite ends: 
\\
{\rm(1)} The flow $v$ is a Hamiltonian flow on $S$ with finite volume. 
\\
{\rm(2)} The flow $v_{\mathrm{end}}$ is a Hamiltonian flow on the compact surface $S_{\mathrm{end}}$.  
%
\end{main}

The present paper consists of four sections.
In the next section, as preliminaries, we introduce fundamental concepts.
In \S 3, we topologically characterize Hamiltonian flows on unbounded surfaces of ``compact type''. 
The characterization implies the topological complete invariance of the finite union of centers, multi-saddles, and their separatrices as a directed surface graph. 
Contrary, under finite volume assumption, any Hamiltonian flows on unbounded surfaces of ``compact type'' can be embedded in Hamiltonian flows on compact surfaces. 
Moreover, we characterize a Hamiltonian flow with sectored ends. 
In the final section, we demonstrate the necessity of degeneracy of connected components of the singular point set, isolated properties of singular points, and non-existence of parabolic sectors for characterization of Hamiltonian flows. 

\section{Preliminaries}\label{sec:prel}

\subsection{Notion of topology and combinatrics}
%
%

By a {\bf surface}, we mean a paracompact two-dimensional manifold.
A {\bf Reeb graph} of a smooth function $f \colon S \to \R$ is the quotient space $S /\sim$ endowed with the quotient topology, where $x \sim y$ if there are a number $c \in \R$ and a connected component of $f^{-1}(c)$ containing $x$ and $y$.

\subsubsection{Curves and arc}

A {\bf curve} (or {\bf arc}) on a surface $S$ is a continuous mapping $C: I \to S$ where $I$ is a non-degenerate connected subset of a circle $\mathbb{S}^1$.
An orbit arc is an arc contained in an orbit.
A curve is {\bf simple} if it is injective.
We also denote by $C$ the image of a curve $C$.
Denote by $\partial C := C(\partial I)$ the boundary of a curve $C$, where $\partial I$ is the boundary of $I \subset \mathbb{S}^1$. Put $\mathrm{int} C := C \setminus \partial C$.
A simple curve is a {\bf simple closed curve} if its domain is $\mathbb{S}^1$ (i.e. $I = \mathbb{S}^1$).

\subsubsection{Directed topological graphs}
A {\bf directed topological graph} is a topological realization of a 1-dimensional simplicial complex with a directed structure on edges. 
A {\bf directed cycle} in a directed topological graph is an embedded cycle whose edges are oriented in the same direction. 
A directed topological graph is a {\bf directed surface graph} if it is realized in a surface.

\subsubsection{End completion of a topological space}
Consider the direct system $\{K_\lambda\}$ of compact subsets of a topological space $X$ and inclusion maps such that the interiors of $K_\lambda$ cover $X$.  
There is a corresponding inverse system $\{ \pi_0( X - K_\lambda ) \}$, where $\pi_0(Y)$ denotes the set of connected components of a space $Y$. 
Then the set of {\bf ends }of $X$ is defined to be the inverse limit of this inverse system. 
Notice that $X$ has one end $x_{\mathcal{U}}$ for each sequence $\mathcal{U} := (U_i)_{i \in \mathbb{Z}_{>0}}$ with $U_i \supseteq U_{i+1}$ such that $U_i$ is a connected component of $X - K_{\lambda_i}$ for some $\lambda_i$. 
Considering the disjoint union $\bm{X_{\mathrm{end}}}$ of $X$ and  $\{ \pi_0( X - K_\lambda ) \}$ as a set, a subset $V$ of the union $X_{\mathrm{end}}$ is an open \nbd of an end $x_{\mathcal{U}}$ if there is some $i \in \mathbb{Z}_{>0}$ such that $U_i \subseteq V$. 
Then the resulting topological space $X_{\mathrm{end}}$ is called the {\bf end completion} (or end compactification) of $X$. 
Note that the end completion is not compact in general. 
From Theorem~3~\cite{richards1963classification}, any connected surfaces of finite genus are  homeomorphic to the resulting surfaces from closed surfaces by removing closed totally disconnected subsets. 
Therefore the end compactification of a connected surface of finite genus is a closed surface and the end compactification of a connected surface with finitely many boundary components and finite genus is a compact surface.

\subsection{Notion of dynamical systems}
%
By a {\bf flow}, we mean a continuous $\mathbb{R}$-action on a surface. 
Let $v \colon  \R \times S \to S$ be a flow on a compact surface $S$. 
Then $v_t := v(t, \cdot)$ is a homeomorphism on $S$. 
For $t \in \R$, define $v_t : S \to S$ by $v_t := v(t, \cdot )$.
For a point $x$ of $S$, we denote by $O(x)$ the orbit of $x$ (i.e. $O(x) := \{ v_t(x) \mid t  \in \R \}$). 
An {\bf orbit arc} is an arc contained in an orbit. 
A subset of $S$ is said to be {\bf invariant} (or saturated) if it is a union of orbits. 
The saturation $\mathrm{Sat}_v(A) = v(A)$ of a subset $A \subseteq S$ is the union of orbits intersecting $A$. 
A point $x$ of $S$ is {\bf singular} if $x = v_t(x)$ for any $t \in \R$ and is {\bf periodic} if there is a positive number $T > 0$ such that $x = v_T(x)$ and  $x \neq v_t(x)$ for any $t \in (0, T)$. 
A point is {\bf closed} if it is either singular or periodic. 
Denote by $\mathop{\mathrm{Sing}}(v)$ (resp. $\mathop{\mathrm{Per}}(v)$) the set of singular (resp. periodic) points. 
A point is {\bf wandering} if there are its neighborhood $U$ and a positive number $N$ such that $v_t(U) \cap U = \emptyset$ for any $t > N$. Then such a neighborhood is called a {\bf wandering domain}. 
A point is {\bf non-wandering} if it is not wandering (i.e. for any its neighborhood $U$ and for any positive number $N$, there is a number $t \in \mathbb{R}$ with $|t| > N$ such that $v_t(U) \cap U \neq \emptyset$).

For a point $x \in S$, define the $\omega$-limit set $\omega(x)$ and the $\alpha$-limit set $\alpha(x)$ of $x$ as follows: $\omega(x) := \bigcap_{n\in \mathbb{R}}\overline{\{v_t(x) \mid t > n\}} $, $\alpha(x) := \bigcap_{n\in \mathbb{R}}\overline{\{v_t(x) \mid t < n\}} $. 
For an orbit $O$, define $\omega(O) := \omega(x)$ and $\alpha(O) := \alpha(x)$ for some point $x \in O$.
Note that an $\omega$-limit (resp. $\alpha$-limit) set of an orbit is independent of the choice of a point in the orbit. 
A {\bf separatrix} is a non-singular orbit whose $\alpha$-limit or $\omega$-limit set is a singular point.
The orbit space $\bm{S/v}$ of $v$ is a quotient space $S/\sim$ defined by $x \sim y$ if $O(x) = O(y)$ (resp. $\overline{O(x)} = \overline{O(y)}$). 
Notice that an orbit space $S/v$ is the set $\{ O(x) \mid x \in S \}$ as a set.

\subsubsection{Resulting flow by the end completion}
For a flow $v$ on a surface $S$ of with finitely many boundary components and finite genus, considering ends to be singular points, we obtain the resulting flow $\bm{v_{\mathrm{end}}}$ on the surface $S_{\mathrm{end}}$ which is a union of compact surfaces. 

\subsubsection{Open flow boxes, open periodic annuli, and open transverse annuli boxes}
An arc is an {\bf orbit arc} if it is contained in an orbit. 
An open disk $U$ is a {\bf trivial flow box} if there is a homeomorphism $h \colon (0,1)^2 \to U$ such that the image $h((0,1) \times \{ t \})$ for any $t \in (0,1)$ is an open orbit arc as in the left of Figure~\ref{flow-boxes}. 
\begin{figure}[t]
\begin{center}
\includegraphics[scale=0.35]{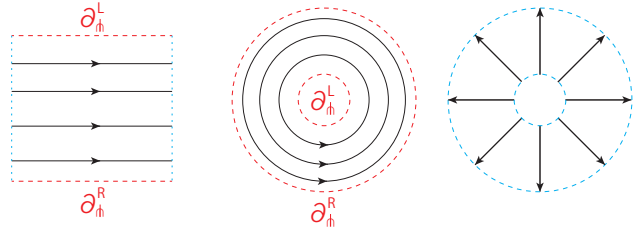}
\end{center}
\caption{An open trivial flow box, an open periodic annulus, and a transverse annulus}
\label{flow-boxes}
\end{figure}
If $h$ can be extend to $h \colon (0,1) \times [0,1] \to U$, then $h((0,1) \times \{ 0 \})$ and $h((0,1) \times \{ 1 \})$ are called {\bf transverse boundary components} of the trivial flow box $U$. 

An invariant annulus $A$ is an {\bf open periodic annulus} if there is a homeomorphism $h \colon \{ (x, y) \mid 1< x^2 + y^2 < 2 \} \to A$ such that the image $h(\{ (x, y) \mid x^2 + y^2 = t \})$ for any $t \in (1,2)$ is a periodic orbit as in the middle of Figure~\ref{flow-boxes}. 
If $h$ can be extend to $h \colon \{ (x, y) \mid 1 \geq x^2 + y^2 \geq 2 \} \to A$, then $h(\{ (x, y) \mid x^2 + y^2 = 1 \})$ and $h(\{ (x, y) \mid x^2 + y^2 = 2 \})$ are called {\bf transverse boundary components} of the open periodic annulus $A$. 

An invariant annulus $A$ is an {\bf open transverse annulus} if there is a homeomorphism $h \colon \{ (x, y) \mid 1< x^2 + y^2 < 2 \} \to A$ such that the image $h(\{ (r \cos \theta, r \sin \theta) \mid 1 < r < \sqrt{2} \})$ for any $\theta \in [0,2\pi)$ is an orbit arc as in the right of Figure~\ref{flow-boxes}.

\subsubsection{Sectors of singular points}

A {\bf parabolic sector} is topologically equivalent to an open trivial flow box with the point $(\pm \infty, 0)$ and a hyperbolic (resp. elliptic) sector is topological equivalent to a Reeb component with the point $(\infty, 0)$ (resp. $(- \infty, 0)$) as in Figure~\ref{sectors_all}.
A separatrix is a {\bf boundary of a hyperbolic sector} if its orbit arc is a boundary component of a hyperbolic sector. 
In other words, a boundary of a hyperbolic sector is the image of a boundary component of the Reeb component with the point $(\infty, 0)$ via a topological conjugacy. 
An isolated singular point $x$ is {\bf finitely sectored} if either it is a center or there is an open \nbd of $x$ which is an open disk and is a finite union of $x$ and sectors such that a pair of distinct sectors intersects at most two orbit arcs. 
A finitely sectored singular point $x$ is {\bf trivial} if it is a center. 
Note that a flow with sectored singular points has at most finitely many singular points. 

\subsubsection{Border separatrices and virtually border separatrices}

A connected component of the union of open elliptic sectors of a singular point is called a {\bf maximal open elliptic sector}. 
Note that a maximal open elliptic sector $U$ is invariant and has the $\omega$- and $\alpha$-limit set $\{ x \}$ of any point in $U$ such that the orbit closure $\overline{O(y)} = O(y) \sqcup \{ x \}$ for any point $y \in U$ bounds an invariant flow box in $U$. 
For a finitely sectored singular point $x$, a nonempty open parabolic sector is a {\bf maximal open parabolic sector} if it is an open parabolic sector contained in the complement of the union of maximal elliptic open sectors of $x$ and hyperbolic sectors of $x$ which is maximal with respect to the inclusion order. 
A separatrix is a {\bf hyperbolic} (resp. {\bf elliptic}, {\bf parabolic}) {\bf border separatrix} if it is contained in the boundary of a hyperbolic sector (resp. a maximal open elliptic sector, a maximal open parabolic sector). 
A separatrix is a {\bf border separatrix} if it is either hyperbolic, elliptic, or parabolic border separatrix. 
%
%
A non-recurrent orbit is a {\bf virtually border separatrix} if it is a border separatrix of the resulting flow $v_{\mathrm{end}}$. 

\subsubsection{Flows of weakly finite type}

A flow is of {\bf weakly finite type} if it has sectored singular points, there are at most finitely many limit cycles, and any recurrent orbits are closed (i.e. $S = \Cv \sqcup \mathrm{P}(v)$).
%
%
%
%
For a flow of weakly finite type on a surface, denote by $\bm{\mathop{\mathrm{BD}_+(v)}}$ the union of singular points, limit cycles, one-sided periodic orbits, and virtually border separatrices. 
The invariant subset $\mathop{\mathrm{BD}_+(v)}$ corresponds to the original one for a flow of weakly finite type on a compact surface because of \cite[Lemma~7.8]{yokoyama2017decompositions}. 
Moreover, we have the following description. 

\begin{lemma}\label{lem:bd1}
The following statements hold for a flow $v$ of weakly finite type on a compact surface $S$: 
\\
{\rm(1)} The invariant subset $\mathop{\mathrm{BD}_+(v)}$ is the finite union of singular points, limit cycles, one-sided periodic orbits, and border separatrices. 
\\
{\rm(2)} Any connected components of the complement $S - \mathop{\mathrm{BD}}_+(v)$ are either invariant open trivial flow boxes, invariant open transverse annuli, invariant open periodic annuli, or periodic tori. 
\end{lemma}

\begin{proof}
By original definitions of $\mathop{\mathrm{BD}_+(v)}$ and $\mathop{\mathrm{BD}_+(v_{\mathrm{end}})}$ in \cite{yokoyama2017decompositions}, since any singular points of $v$ are sectored, we have  $\mathop{\mathrm{BD}_+(v)} = \mathop{\mathrm{BD}_+(v_{\mathrm{end}})}$. 
By \cite[Lemma~7.8]{yokoyama2017decompositions}, the invariant subset $\mathop{\mathrm{BD}_+(v)} = \mathop{\mathrm{BD}_+(v_{\mathrm{end}})}$ is the finite union of singular points, limit cycles, one-sided periodic orbits, and border separatrices. 
\end{proof}

\subsubsection{Circuit}
By a {\bf cycle}, we mean a periodic orbit.
By a {\bf non-trivial circuit}, we mean either a cycle or a continuous image of a directed cycle which is a graph but not a singleton, whose orientations of edges correspond to the directions of orbits, and which is the union of separatrices and finitely many singular points.
In other words, a non-trivial circuit that is not a cycle is a directed path as a graph whose initial point is also terminal.
Note that there are non-trivial circuits with infinitely many edges and that any non-trivial non-periodic circuit contains non-recurrent orbits.
A non-trivial circuit $\gamma$ is a {\bf limit circuit} if it is the $\alpha$-limit or $\omega$-limit set of a point outside of $\gamma$. 

\subsubsection{Transversality for continuous flows on surfaces}
Notice that we can define transversality using tangential spaces of surfaces because each flow on a compact surface is topologically equivalent to a $C^1$-flow by Gutierrez's smoothing theorem~\cite{gutierrez1978structural}.
However, to modify transverse arcs explicitly, we define transversality immediately as follows.  

A curve $C$ is {\bf transverse to} $v$ at a point $p \in \mathrm{int} C$ if there are a small neighborhood $U$ of $p$ and a homeomorphism $h:U \to [-1,1]^2$ with $h(p) = 0$ such that $h^{-1}([-1,1] \times \{t \})$ for any $t \in [-1, 1]$ is an orbit arc and $h^{-1}(\{0\} \times [-1,1]) = C \cap U$.
A curve $C$ is {\bf transverse to} $v$ at a point $p \in \partial C \cap \partial S$ (resp. $p \in \partial C \setminus \partial S$) if there are a small neighborhood $U$ of $p$ and a homeomorphism $h:U \to [-1,1] \times [0,1]$ (resp. $h:U \to [-1,1]^2$) with $h(p) = 0$ such that $h^{-1}([-1,1] \times \{t \})$ for any $t \in [0, 1]$ (resp. $t \in [-1, 1]$) is an orbit arc and $h^{-1}(\{0\} \times [0,1]) = C \cap U$ (resp. $h^{-1}(\{0\} \times [-1,1]) = C \cap U$).
A simple curve $C$ is {\bf transverse to} $v$ if so is it at any point in $C$.
A simple curve $C$ is {\bf transverse to} $v$ is called a {\bf transverse arc}. 
A simple closed curve is a {\bf closed transversal} if it transverses to $v$. 

\subsection{Fundamental notion of Hamiltonian flows on a compact surface}
A $C^r$ vector field $X$ for any $r \in \Z_{\geq0}$ on an orientable surface $S$ is {\bf Hamiltonian} if there is a $C^{r+1}$ function $H \colon S \to \mathbb{R}$ such that $dH= \omega(X, \cdot )$ as a one-form, where $\omega$ is a volume form of $S$.
In other words, locally the Hamiltonian vector field $X$ is defined by $X = (\partial H/ \partial x_2, - \partial H/ \partial x_1)$ for any local coordinate system $(x_1,x_2)$ of a point $p \in S$.
A flow is {\bf Hamiltonian} if it is topologically equivalent to a flow generated by a Hamiltonian vector field.
Note that a volume form on an orientable surface is a symplectic form.
By \cite[Theorem 3]{cobo2010flows}, any singular points of a Hamiltonian vector field with finitely many singular points on a compact surface is either a center or a multi-saddle. 

To characterize compact-like behavior of Hamiltonian dynamics, for a flow $v$ on a surface $S$, denote by $\bm{\mathop{\mathrm{Sing}}(v)_\infty}$ the set of singular points which are neither centers nor multi-saddles and put $\bm{S_{\mathrm{unbd}}} := S - \mathop{\mathrm{Sing}}(v)_\infty$ 
It is known that a $C^r$ ($r \geq 1$) Hamiltonian vector field on a compact surface is structurally stable with respect to the set of $C^r$ Hamiltonian vector fields if and only if both each singular point is nondegenerate and each separatrix is self-connected (see  \cite[Theorem 2.3.8, p. 74]{ma2005geometric}).

%

\subsubsection{Extended orbit spaces of Hamiltonian flows on unbounded surfaces}
A non-singular point $x$ is a {\bf non-Hausdorff point} if there is a non-singular point $y \in S$ with $y \notin \overline{O(x)}$ and $x \notin \overline{O(y)}$ such that there are no disjoint invariant \nbds $U_x$ and $U_y$ of $x$ and $y$ respectively. 
Then the pair $x$ and $y$ is called a {\bf non-Hausdorff pair}. 
Define an equivalence relation $\sim_{\mathrm{nH}}$ on $S$ as follows: 
$x \sim_{\mathrm{nH}} y$ if either $O(x) = O(y)$ or the pair of $x$ and $y$ is a non-Hausdorff pair. 
For any non-Hausdorff point $x$, denote by $[x]_{\mathrm{nH}}$ the equivalent class of $x$.
For a flow $v$ of weakly finite type on a surface $S$ with finitely many boundary components, finite genus, and finite ends, the {\bf extended orbit} $O_{\mathrm{ex}}(x)$ of any non-Hausdorff point $x$ is defined as follows: 
\[
O_{\mathrm{ex}}(x) := [x]_{\mathrm{nH}} \cup \left( \bigcup_{y \in [x]_{\mathrm{nH}}} \alpha(y) \cup \omega(y) \right) \subset S
\]
For any point in the complement of the union of extended orbits of non-Hausdorff points, its extended orbit is defined as its orbit. 
Define an equivalence relation $\sim_{\mathrm{ex}}$ on $S$ as follows: 
$x \sim_{\mathrm{ex}} y$ if there is an extended orbit that contains $x$ and $y$. 
Then the quotient space $S/\sim_{\mathrm{ex}}$ is called the {\bf extended orbit space} and is denoted by $\bm{S/v_{\mathrm{ex}}}$. 
Notice that the extended orbit space of $v$ is a quotient space of the orbit space of $v$. 
The extended orbit space of a flow of weakly finite type on a surface with finitely many boundary components, finite genus, and finite ends is a generalization of the extended orbit space on a compact surface. 
In particular, the extended orbit space of a Hamiltonian flows with finitely many singular points on a compact surface in the sense of this paper corresponds to the extended orbit space in the sense of one in \cite{yokoyama2021ham}. 
Indeed, the extended orbit space $S/v_{\mathrm{ex}}$ is the quotient space $S/\sim_{\mathrm{ms}}$, where $x \sim_{\mathrm{ms}} y$ if there is either an orbit or a multi-saddle connection that contains $x$ and $y$. 

\section{Characterization of Hamiltonian flows on unbounded surfaces}

We characterize isolated singular points for Hamiltonian flows as follows. 

\begin{lemma}\label{lem:char_isolated_sing}
Any isolated singular point for a Hamiltonian flow on a surface is either a center or a multi-saddle. 
\end{lemma}

\begin{proof}
Let $v$ be a Hamiltonian flow on a surface $S$ and $x$ an isolated singular point.
By the existence of a Hamiltonian, there are no closed transversals. 
This implies the non-existence of non-closed recurrent orbits and so $S = \Cv \sqcup \mathrm{P}(v)$. 
Since the existence of limit circuits implies the existence of closed transversals outside of the limit circuits, by a generalization of the Poincar\'e--Bendixson theorem to $v_{\mathrm{end}}$, the $\omega$-limit set and $\alpha$-limit set of a non-recurrent point with respect to $v_{\mathrm{end}}$ are singular points. 
Taking the double of $S$ if necessary, we may assume that $x$ is outside of the boundary $\partial S$. 
Therefore there is a simple closed curve $\gamma$ which bounds a closed disk $D$ with $D \cap \Sv = \{ x \}$ and which is transverse to $v$ except finitely many points (cf. \cite[Lemma 3.1]{kibkalo2021topological}). 
From $D - \{ x \} \subseteq S - \Sv$, we obtain that the set difference $D  \setminus H^{-1}(H(x)) = (D  - \{ x \}) \setminus H^{-1}(H(x))$ is dense in $D$. 
We may assume that $x$ is not a center. 

We claim that $D$ contains no periodic orbits. 
Indeed, assume that there is a periodic orbit $O \subset D$. 
Since any closed disk can be embedded in a sphere, by Jordan-Schoenflies theorem, the periodic orbit bounds a closed disk $B \subseteq D$. 
Then the restriction $v|_B$ is a Hamiltonian flow with finitely many singular points on the compact surface $B$. 
\cite[Theorem 3]{cobo2010flows} implies that every singular point for $v|_B$ is either a center or a multi-saddle. 
From Poincar\'e-Hopf theorem, by $D \cap \Sv = \{ x \}$, $B \subseteq D$ contains a center and so $x$ is a center. 

For any point $y \in D \setminus H^{-1}(H(x))$, since the $\omega$-limit set and $\alpha$-limit set of any non-recurrent point with respect to $v_{\mathrm{end}}$ are singular points, we obtain $O^-(y) \not\subseteq D$ and $O^+(y) \not\subseteq D$.  
Let $x_1, x_2, \ldots , x_k$ be the tangencies of $\partial D$. 
Put $C : = \bigcup_{i=1}^k H^{-1}(H(x_i)) \cup H^{-1}(H(x))$. 
Since $D - \{x\} \subseteq S - \Sv$, the existence of the Hamiltonian implies that the set difference $D \setminus C = (D - \{x\}) \setminus C$ is dense in $D$. 
Replace the tangencies $x_1, x_2, \ldots , x_k$ into singular points, the resulting flow is denoted by $v'$. 
Taking the double $D \cup -D$ of $D$ and considering the time-reversing flow $-v'$ on $-D$, the resulting space $D \cup -D$ is a sphere and the resulting flow $v'_{D \cup -D}$ satisfies that $v'_{D \cup -D}(\gamma \setminus H^{-1}(H(C))) \subset \mathop{\mathrm{Cl}}(v'_{D \cup -D})$ and so that $D \cup -D = \overline{\mathop{\mathrm{Cl}}(v_{D \cup -D})}$. 
Therefore $v'_{D \cup -D}$ is a non-wandering flow with finitely many singular points on a sphere. 
By \cite[Theorem 3]{cobo2010flows}, the singular point $x$ is a multi-saddle with respect to $v'_{D \cup -D}$ and so $v$. 
\end{proof}

We have the following topological characterization of Hamiltonian flows on punctured surfaces. 

\begin{lemma}\label{lem:02}
The following are equivalent for a flow $v$ with sectored singular points on a compact surface $S$: 
\\
{\rm(1)} The restriction $v|_{S_{\mathrm{unbd}}}$ is Hamiltonian. 
\\
{\rm(2)} The flow $v$ is a flow of weakly finite type, the restriction $S_{\mathrm{unbd}}/v_{\mathrm{ex}}$ is a finite directed graph without directed cycles, and $S$ is orientable.  

In any case, $\mathop{\mathrm{BD}_+(v)}$ consists of finitely many singular points and border separatrices, and any connected components of the complement $S - \mathop{\mathrm{BD}}_+(v)$ are either open trivial flow boxes or periodic annuli. 
\end{lemma}

\begin{proof}
We may assume that $S$ is connected. 
Suppose that the restriction $v|_{S_{\mathrm{unbd}}}$ is Hamiltonian. 
Since $v|_{S_{\mathrm{unbd}}}$ is Hamiltonian, there are no closed transversals in $S_{\mathrm{unbd}}$. 
The non-existence of closed transversals on $S_{\mathrm{unbd}}$ implies the non-existence of limit circuits and non-closed recurrent orbits on $S_{\mathrm{unbd}}$. 
Then $S_{\mathrm{unbd}} \subseteq \Cv \sqcup \mathrm{P}(v)$. 
By the finiteness of $\mathop{\mathrm{Sing}}(v)_\infty$, the surface $S = S_{\mathrm{unbd}} \sqcup \mathop{\mathrm{Sing}}(v)_\infty$ is orientable and there are no closed transversals in $S$. 
Then $S = \Cv \sqcup \mathrm{P}(v)$. 
Therefore $v$ is a flow of weakly finite type. 
Lemma~\ref{lem:bd1} implies that $\mathop{\mathrm{BD}_+(v)}$ consists of finitely many orbits and the complement $S - \mathop{\mathrm{BD}_+(v)}$  is a finite disjoint union of invariant open trivial flow boxes and invariant open periodic annuli. 
Then the quotient space $(S - \mathop{\mathrm{BD}}_+(v))/v_{\mathrm{ex}} = (S - \mathop{\mathrm{BD}}_+(v))/v$ is a finite disjoint union of open intervals. 
Since non-Hausdorff pairs have the same values, the restriction $S_{\mathrm{unbd}}/v_{\mathrm{ex}}$ is a finite directed graph. 
The existence of the Hamiltonian $H \colon S_{\mathrm{unbd}} \to \R$ implies that the induced function $H_{\mathrm{ex}} \colon S_{\mathrm{unbd}}/v_{\mathrm{ex}} \to \R$ by the Hamiltonian $H$ is well-defined. 
Therefore the directed graph $S_{\mathrm{unbd}}/v_{\mathrm{ex}}$ contains no directed cycles and so is a finite directed graph without directed cycles.

Conversely, suppose that $v$ is a flow of weakly finite type, the restriction $S_{\mathrm{unbd}}/v_{\mathrm{ex}}$ is a finite directed graph without directed cycles, and $S$ is orientable. 
Then there are no closed transversals. 
By Lemma~\ref{lem:bd1}, the singular point set $\Sv$ consists of at most finitely many singular points,  and the non-existence of closed transversals implies that any connected components of the complement $S - \mathop{\mathrm{BD}}_+(v)$ are either open trivial flow boxes or periodic annuli. 
In particular, there are no limit circuits. 
By a generalization of the Poincar\'e--Bendixson theorem~(cf. \cite[Theorem 2.6.1]{nikolaev1999flows}), the $\omega$-limit set and $\alpha$-limit set of a non-recurrent point are singular points. 
Then the orbit space $(S - \mathop{\mathrm{BD}}_+(v))/v$ is a finite disjoint union of directed open intervals induced by the orientation of $S$. 
From the non-existence of limit cycles, Lemma~\ref{lem:bd1} implies that the difference $\mathop{\mathrm{BD}}_+(v) - \Sv$ consists of finitely many one-sided periodic orbits and border separatrices. 
By the flow box theorem for a continuous flow on a compact surface (cf. Theorem 1.1, p.45\cite{aranson1996introduction}), there is a flow box for an orbit of the difference $\mathop{\mathrm{BD}}_+(v) - \Sv$. 
Define a graph $G= (V, D)$ as follows: 
Define a binary relation $\sim$ on the difference $\mathop{\mathrm{BD}}_+(v) - \Sv$ by $x \sim y$ if there are a connected component $U$ of the complement $S - \mathop{\mathrm{BD}}_+(v)$ and a letter $\sigma \in \{ R, L \}$ such that $x, y \in  \partial^\sigma U$. 
Denote by $\sim_V$ the transitive closure of $\sim$.  
The vertex set $V$ is the quotient space $(\mathop{\mathrm{BD}}_+(v) - \Sv)/\sim_V$.  
A directed edge from a vertex $\gamma$ to a vertex $\mu$ exist if there is a connected component $U$ of the complement $S - \mathop{\mathrm{BD}}_+(v)$ such that $\gamma \cap \partial^R U \neq \emptyset$ and $\mu \cap \partial^L U \neq \emptyset$. 
By the orientability of $S$, since the extended orbit space  $S/v_{\mathrm{ex}}$ has no directed cycles as a directed graph, 
the head of a directed edge does not correspond to the tail and so that there are no directed cycles in $G$. 
Therefore $G$ is a finite simple acyclic directed graph. 
Then we can define a height function of $G$ and so the induced height function $h \colon S_{\mathrm{unbd}} \to \R$ which is constant along each orbit. 

We claim that we can construct the Hamiltonian whose generating flow is topological equivalent to $v|_{S_{\mathrm{unbd}}}$ by modifying $h$. 
Indeed, let $D'_1, \ldots ,D'_k$ be the connected components of $S - \mathop{\mathrm{BD}_+(v)}$, each of which is either an invariant open trivial flow box or an invariant open periodic annulus. 
Since the difference $U - \{x \}$ of an open disk $U$ containing a non-trivial finitely sectored singular point $x$ can be identified with an open annulus which can be constructed by finitely many sectors as in  Figure~\ref{sectors_all} pasting the transverse boundary components of the sectors, the difference $S_{\mathrm{unbd}}$ can be obtained by pasting finitely many domains $D_1, \ldots , D_k$ along $\mathop{\mathrm{BD}_+(v)} - \mathop{\mathrm{Sing}}(v)_\infty$ such that each domain $D_i = \overline{D'_i} \setminus \mathop{\mathrm{Sing}}(v)_\infty$ and that $D_i \setminus \mathop{\mathrm{BD}_+(v)} = D'_i$. 
Moreover, the intersection $D_i \cap \mathop{\mathrm{BD}_+(v)}$ consists of finitely many border separatrices $O_{i,1}, \ldots ,O_{i,k_i}$ in $\mathrm{P}(v)$ and finitely many centers and multi-saddles. 
For any orbit $O_{i,j}$, there are connected components $D'$ and $D''$ of $S - \mathop{\mathrm{BD}_+(v)}$ such that $(D' \cup D'') \sqcup O_{i,j}$ is an open \nbd of $O_{i,j}$.
This implies that there are flow boxes $U_{i,j}$ and $V_{i,j}$ containing $O_{i,j}$ with $\overline{U_{i,j}} \setminus \Sv \subset V_{i,j}$ such that $V_{i,j}$ are pairwise disjoint and that $U_{i,j}$ can be identified with a flow box $(0,1) \times (-1,1)$ with orbit arcs $\{x \} \times (-1,1)$ for any $x \in (0,1)$ such that $O_{i,j}$ is $\{ 0 \} \times (-1,1)$. 
We identify $U_{i,j} \sqcup \partial^R U_{i,j} \sqcup \partial^L U_{i,j}$ with a flow box $(0,1) \times [-1,1]$ as in Figure~\ref{flowbox_oij}. 
\begin{figure}[t]
\begin{center}
\includegraphics[scale=0.45]{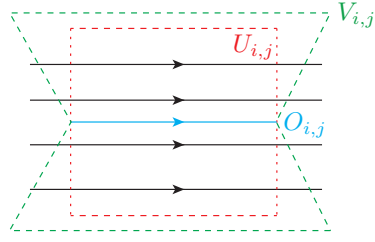}
\end{center}
\caption{A flow box $V_{i,j}$}
\label{flowbox_oij}
\end{figure}
Define a smooth Hamiltonian function $H$ on $S_{\mathrm{unbd}}$ whose Reeb graph is homeomorphic to the Reeb graph of $h$ such that all derivatives of $H|_{\partial^R D'_i \cup \partial^L D'_i}$ are zero, that the flows generated by $H|_{D'_i}$ correspond to $v|_{D'_i}$ up to topological equivalence, and that $H|_{V_{i,j}}$ depends only on $y$ and strictly increasing except $0$ as in the left of Figure~\ref{flowbox_oij_ham}. 
\begin{figure}[t]
\begin{center}
\includegraphics[scale=0.425]{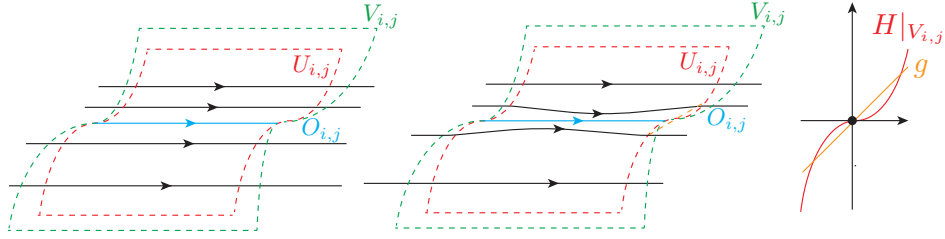}
\end{center}
\caption{Left and middle, a deformation of the flow on the flow box $U_{i,j}$; right, the functions $H|_{V_{i,j}}$ and $g_{i,j}$.}
\label{flowbox_oij_ham}
\end{figure}
Replacing $U_{i,j}$ with a small rectangle if necessary, we can take a linear function $g_{i,j} \colon [-1,1] \to \R$ with $g_{i,j}|_{\{-1,1\}} = H|_{\{-1,1\}}$ as in the right of Figure~\ref{flowbox_oij_ham}. 
Define a smooth bump function $\varphi_{i,j} \colon \R \to [0,1]$ with $\R - (0,1) = \varphi_{i,j}^{-1}(0)$ and $[1/4,3/4] \subset \varphi_{i,j}^{-1}(1)$ and a smooth bump function $\psi_{i,j} \colon \R \to [0,1]$ with $\R - (-1,1) = \varphi_{i,j}^{-1}(0)$ and $[-3/4,3/4] \subset \psi_{i,j}^{-1}(1)$, and $\widetilde{H} \colon S_{\mathrm{unbd}} \to \R$ by $\widetilde{H} = H$ on $S - (\Sv_\infty \cup \overline{U_{i,j}})$ and $\widetilde{H}(x,y) := (1 - \varphi_{i,j}(x)\psi_{i,j}(y)) H(y) + \varphi_{i,j}(x)\psi_{i,j}(y) g(y)$ on $U_{i,j} = (0,1) \times (-1, 1)$. 
By definition, the function $\widetilde{H}$ is smooth such that $\widetilde{H}|_{U_{i,j}}(x, \cdot)$ for $x \in (0,1)$ is strictly increasing as in the middle of Figure~\ref{flowbox_oij_ham}. 
Therefore the Hamiltonian flow generated by $\widetilde{H}$ is topologically equivalent to $v|_{S_{\mathrm{unbd}}}$. 
\end{proof}

In the proof of the previous lemma, the finiteness of singular points is essential (see Proposition~\ref{prop:ex} for details). 

\subsection{Proof of Theorem~\ref{main:01}}
Let $v$ be a flow with sectored ends and finitely many singular points on an orientable surface $S$ with finitely many boundary components, finite genus, and finite ends. 
Suppose that $v$ is a flow of weakly finite type, the restriction $S/v_{\mathrm{ex}}$ is a finite directed graph without directed cycles.  
Since each end is sectored, the flow $v_{\mathrm{end}}$ is also of weakly finite type. 
From Lemma~\ref{lem:02}, assertion {\rm(1)} holds. 
By Lemma~\ref{lem:bd1}, the invariant subset $\mathop{\mathrm{BD}_+(v_{\mathrm{end}})}$ is the finite union of singular points and border separatrices. 
Since border separatrices of $v_{\mathrm{end}}$ are virtually border separatrices of $v$, the invariant subset $\mathop{\mathrm{BD}_+(v)}$ is the finite union of singular points and virtually border separatrices, and is the union of the vertices of the finite directed graph $S/v_{\mathrm{ex}}$. 
From Lemma~\ref{lem:02}, any connected components of the complement $S_{\mathrm{end}} - \mathop{\mathrm{BD}}_+(v_{\mathrm{end}}) = S - \mathop{\mathrm{BD}}_+(v)$ are either invariant open trivial flow boxes or invariant open periodic annuli. 
Therefore each non-Hausdorff point in $S/v$ is either a multi-saddle or a virtually border separatrix. 
Since the Reeb graph of the Hamiltonian which generates a flow which is topological equivalent to $v$ is a quotient space of the orbit space $S/v$, by identifying non-Hausdorff points of the Reeb graph, the resulting space is the finite directed graph $S/v_{\mathrm{ex}}$. 

Conversely, suppose that $v$ is Hamiltonian. 
By Lemma~\ref{lem:char_isolated_sing}, any singular points on $S$ are either multi-saddles or centers.
Since each end is sectored, the singular point set $\mathop{\mathrm{Sing}}(v_{\mathrm{end}})$ is finite and so any singular points are finitely sectored. 
Since the existence of limit circuits implies the existence of closed transversals outside of the limit circuits, by a generalization of the Poincar\'e--Bendixson theorem to $v_{\mathrm{end}}$, the $\omega$-limit set and $\alpha$-limit set of a non-recurrent point with respect to $v_{\mathrm{end}}$ are singular points. 
%
%
Then the flow $v_{\mathrm{end}}$ is of weakly finite type. 
By Lemma~\ref{lem:02}, assertion {\rm(2)} holds.


%


%

\subsection{Proof of Corollary~\ref{main_cor:01}}

Let $v$ be a Hamiltonian flow with sectored ends and finitely many singular points on an orientable surface $S$ with finitely many boundary components, finite genus, and finite ends and $\mathop{\mathrm{BD}}_+(v)$ the finite union of singular points and virtually border separatrices because of Lemma~\ref{lem:bd1}. 
Theorem~\ref{main:01} implies that the complement $S - \mathop{\mathrm{BD}}_+(v)$ consists of invariant annuli and invariant disks. 
%
This means that we can reconstruct $v$ from the finite directed surface graph $\mathop{\mathrm{BD}}_+(v)$ by filling invariant periodic annuli (resp. invariant trivial blow boxes) into invariant annuli (resp. invariant disks). 

\subsection{Proof of Theorem~\ref{main:02}}

Let $v$ be a flow with finitely many singular points on an orientable surface $S$ with finitely many boundary components, finite genus, and finite ends. 
By the finite existence of genus and ends, we may assume that $S$ is connected. 

Suppose that $v_{\mathrm{end}}$ is a Hamiltonian flow on the compact surface $S_{\mathrm{end}}$.   
By \cite[Theorem 3]{cobo2010flows}, any singular points are centers or multi-saddles.
Since the compact surface has a finite volume and $S$ is an open subset of $S_{\mathrm{end}}$, the flow $v$ is a Hamiltonian flow on $S$ with finite volume. 

Conversely, suppose that the flow $v$ is a Hamiltonian flow on $S$ with finite volume. 
By the finite existence of genus and ends, the end completion $S_{\mathrm{end}}$ is a compact surface and the resulting flow $v_{\mathrm{end}}$ is a flow with finitely many singular points on $S_{\mathrm{end}}$. 

We claim that $\Pv$ is open. 
Indeed, applying a generalization of the Poincar\'e--Bendixson theorem to $v_{\mathrm{end}}$, the $\omega$-limit set and $\alpha$-limit set of a non-recurrent point with respect to $v_{\mathrm{end}}$ are singular points. 
Fix a periodic point $O$. 
By the flow box theorem for a continuous flow on a compact surface (cf. Theorem 1.1, p.45\cite{aranson1996introduction}), there are finitely many trivial flow boxes whose union is a \nbd of $O$. 
The non-existence of limit cycles and the orientability imply there is a transverse arc containing a point of $O$ whose saturation is a periodic annulus and a \nbd of $O$.

Then the complement $S - \Pv = \Sv \sqcup \mathrm{P}(v)$ is closed. 
We claim that $\mathrm{int} \mathrm{P}(v) = \emptyset$. 
Indeed, assume that $\mathrm{int} \mathrm{P}(v) \neq \emptyset$. 
Fix a point $x \in \mathrm{int} \mathrm{P}(v)$. 
From the finite volume property, any open \nbd of $x$ is not a wandering domain. 
Fix a small open transverse arc $I \subset \mathrm{int} \mathrm{P}(v)$ containing $x$ and a positive number $T >0$ such that $v_T(I) \cap I \neq \emptyset$. 
The existence of the Hamiltonian implies that any points in $v_T(I) \cap I \neq \emptyset$ are periodic, which contradicts that $I \subset \mathrm{int} \mathrm{P}(v)$. 

The finite existence of singular points implies that $S = \overline{\Pv}$. 
By $S_{\mathrm{end}} - \mathop{\mathrm{Sing}}(v_{\mathrm{end}}) = S - \Sv$, we obtain that $S_{\mathrm{end}} = \overline{\mathop{\mathrm{Per}}(v_{\mathrm{end}})}^{S_{\mathrm{end}}}$ and so that $v_{\mathrm{end}}$ is a non-wandering flow with finitely many singular points. 
By \cite[Theorem 3]{cobo2010flows}, any singular points of $v_{\mathrm{end}}$ are centers or multi-saddles. 
This implies that $\mathrm{P}(v) = \mathrm{P}(v_{\mathrm{end}})$ consists of finitely many orbits. 
Since any non-Hausdorff points in the orbit space $S/v$ are separatrices of a multi-saddle, we obtain that $S/v_{\mathrm{ex}}$ is homeomorphic to $S_{\mathrm{end}}/v_{\mathrm{end}}$. 
Applying Theorem~\ref{main:01} to $v$, the extended orbit space $S/v_{\mathrm{ex}}$ is a finite directed graph without directed cycles and so is $S_{\mathrm{end}}/v_{\mathrm{end}}$.  
Since $(S_{\mathrm{end}})_{\mathrm{unbd}} = S_{\mathrm{end}}$, applying Lemma~\ref{lem:02} to $v_{\mathrm{end}}$, we have that $v_{\mathrm{end}}$ is a Hamiltonian flow on the compact surface $S_{\mathrm{end}}$. 

%

\subsection{Characterization of a Hamiltonian flow with sectored ends}

In this subsection, we have the following characterization of a Hamiltonian flow with sectored ends and finitely many singular points on an orientable surface of ``compact type''. 

\begin{theorem}\label{thm:02}
A flow $v$ with finitely many singular points on an orientable surface $S$ with finitely many boundary components, finite genus, and finite ends is a Hamiltonian flow with sectored ends if and only if the flow $v$ satisfies the following four conditions: 
\\
{\rm(1)} The flow $v$ is a flow of weakly finite type. 
\\
{\rm(2)} The extended orbit space $S/v_{\mathrm{ex}}$ is a finite directed graph without directed cycles. 
\\
{\rm(3)} The complement $S - \mathop{\mathrm{BD}_+(v)}$ consists of finitely many invariant open subsets which are periodic annuli and trivial flow boxes. 
\\
{\rm(4)} Any non-Hausdorff points are contained in virtually border separatrices. 
\end{theorem}

The author would like to know whether condition {\rm(4)} in the previous theorem is redundant or not. 
We have the following statements to demonstrate the above characterization. 

\begin{lemma}\label{lem:condition01}
Let $v$ be a flow of weakly finite types on an orientable surface $S$ with finitely many boundary components, finite genus, and finite ends. 
If $S - \mathop{\mathrm{BD}_+(v)}$ consists of finitely many invariant open subsets which are periodic annuli and trivial flow boxes, then $\mathop{\mathrm{BD}_+(v)}$ consists of finitely many orbits. 
\end{lemma}

\begin{proof}
Suppose that $S - \mathop{\mathrm{BD}_+(v)}$ consists of finitely many invariant open subsets which are periodic annuli and trivial flow boxes. 
Let $U_1, \ldots , U_k$ be the connected components of $S - \mathop{\mathrm{BD}_+(v)}$. 
By definition, the invariant subset $\mathop{\mathrm{BD}_+(v)}$ is the union of singular points, limit cycles, one-sided periodic orbits, and virtually border separatrices. 
Assume that there are infinitely many pairwise disjoint orbits $O_i \subseteq \mathop{\mathrm{BD}_+(v)}$ for $i \in \Z_{\geq 0}$. 
Since $v$ is of weakly finite type and $S$ has finitely many boundary components, finite genus, and finite ends, we have the finite existence of singular points, limit cycles, and one-sided periodic orbits. 
Therefore we may assume that any orbits $O_i$ are virtually border separatrices. 
By the finite existence of finitely sectored singular points and ends, we may assume that there are ends $\alpha, \omega \in S_{\mathrm{end}} - S$ such that $\alpha_{v_{\mathrm{end}}}(O_i) = \alpha$ and $\omega_{v_{\mathrm{end}}}(O_i) = \omega$ for any $i \in \Z_{\geq 0}$. 
Since each border separatrix is contained in the boundary component of a sector which is either hyperbolic, elliptic, or parabolic, for any $i \in \Z_{\geq 0}$, there is an invariant subset $U_j$ whose boundary contains $O_i$. 
Then $\bigsqcup_{i \in \Z_{\geq 0}} O_i \subseteq \bigcup_{j=1}^k \partial U_j$. 
By renumbering, we may assume that $\bigsqcup_{i \in \Z_{\geq 0}} O_i \subseteq \partial U_1$. 
Then $\alpha, \omega \in \partial_{S_{\mathrm{end}}} U_1$. 
Put $\mu_i := \{ \alpha, \omega \} \sqcup O_i \sqcup O_{i+1}$. 
Then the unions $\mu_i$ are simple closed curves. 
From the finite existence of genus and ends, by renumbering, we may assume that any simple closed curves $\mu_i$ bound open disks. 
By renumbering, we may assume that the open disks $B_1$, $B_2$, $B_3$ are pairwise disjoint. 
Then the union $B_1 \sqcup O_2 \sqcup B_2 \sqcup O_3 \sqcup B_3$ is an open invariant disk. 
Since $O_2, O_3 \subseteq \partial U_1$ and $U_1 \cap \bigcup_{i=1}^\infty O_i = \emptyset$, the connectivity of $U_1$ implies that $U_1 \subset B_1 \sqcup B_2 \sqcup B_3$. 
Therefore $O_1 \sqcup O_4 \not\subseteq U_1$, which contradicts $\bigsqcup_{i \in \Z_{\geq 0}} O_i \subseteq \partial U_1$. 
\end{proof}

\begin{lemma}\label{lem:condition02}
Let $v$ be a flow of weakly finite type on an orientable surface $S$ with finitely many boundary components, finite genus, and finite ends. 
Suppose that any non-Hausdorff points are contained in virtually border separatrices, the extended orbit space $S/v_{\mathrm{ex}}$ is a finite directed graph without directed cycles, and that $S - \mathop{\mathrm{BD}_+(v)}$ consists of finitely many invariant open subsets which are periodic annuli and trivial flow boxes. 
Then the ends are sectored. 
\end{lemma}

\begin{proof}
Lemma~\ref{lem:condition01} implies that $\mathop{\mathrm{BD}_+(v)}$ consists of finitely many orbits. 
The finiteness of genus and ends implies that $\mathop{\mathrm{BD}_+(v_{\mathrm{end}})}$ consists of finitely many orbits. 
By definition, the invariant subset $\mathop{\mathrm{BD}_+(v_{\mathrm{end}})}$ is a closed subset that contains all virtually border separatrices and which consists of singular points, limit cycles, one-sided periodic orbits, and virtually border separatrices. 
Therefore there are at most finitely many virtually border separatrices. 
Since any non-Hausdorff points are contained in virtually border separatrices, any extended orbit is the finite union of orbits and any non-Hausdorff points are contained in $\mathop{\mathrm{BD}_+(v)}$. 
Because $S/v_{\mathrm{ex}}$ is a finite directed graph without directed cycles, there are no closed transversals of $v$ and the extended orbit space $S/v_{\mathrm{ex}}$ is Hausdorff. 
Since any extended orbits are the inverse images of the canonical quotient map $S \to S/v_{\mathrm{ex}}$, any extended orbits are closed. 
Because the non-existence of closed transversals outside of the limit circuits implies the non-existence of limit circuits, from the non-existence of non-closed recurrent points, by a generalization of the Poincar\'e--Bendixson theorem to $v_{\mathrm{end}}$, the $\omega$-limit set and $\alpha$-limit set of a non-recurrent point with respect to $v_{\mathrm{end}}$ are singular points. 
This implies that any extended orbits of non-periodic points are finite graphs. 
Then $\mathop{\mathrm{BD}_+(v_{\mathrm{end}})}$ is a finite union of singular points, one-sided periodic orbits, and virtually border separatrices, and each connected component of $\mathop{\mathrm{BD}_+(v_{\mathrm{end}})}$ is either a periodic orbit or a finite graph. 
Fix a trivial flow box $U$ which is a connected component of $S - \mathop{\mathrm{BD}_+(v)} = S_{\mathrm{end}} - \mathop{\mathrm{BD}_+(v_{\mathrm{end}})}$. 
The non-existence of limit cycles implies that $\partial_{v_{\mathrm{end}}} U \subseteq \mathop{\mathrm{BD}_+(v)}$ is a finite circuit consisting of singular points and virtually border separatrices. 

We claim that there are points $\alpha, \omega \in S_{\mathrm{end}} - S$ such that $\alpha_{v_{\mathrm{end}}}(y) = \alpha$ and $\omega_{v_{\mathrm{end}}}(y) = \omega$ for any $y \in U$. 
Indeed, assume that there are distinct point $y, y' \in U$ such that $\alpha_{v_{\mathrm{end}}}(y) \neq \alpha_{v_{\mathrm{end}}}(y')$. 
Fix a point $z \in U \cap \partial \{ y'' \in U \mid \alpha_{v_{\mathrm{end}}}(y) \neq \alpha_{v_{\mathrm{end}}}(y'') \}$. 
Since $\partial_{v_{\mathrm{end}}} U \subseteq \mathop{\mathrm{BD}_+(v)}$ consists of finitely many orbits, there is a convergence sequence $(z_i)_{i \in \Z_{>0}}$ to $z$ with $\alpha_{v_{\mathrm{end}}} (z) \neq \alpha_{v_{\mathrm{end}}}(z_i) = \alpha_{v_{\mathrm{end}}}(z_j)$ for any $i,j \in \Z_{>0}$. 
Put $\alpha := \alpha_{v_{\mathrm{end}}}(z)$ and $\alpha' := \alpha_{v_{\mathrm{end}}}(z_i)$. 
Fix a closed transverse arc $T$ such that $z$ is the boundary point of $T$ and that $T$ intersects infinitely many orbits $O(z_i)$. 
Taking a subsequence of $(z_i)_{i \in \Z_{>0}}$, any orbits $O(z_i)$ intersects $T$ and the first intersections $z_i'$ of $O(z_i) \cap T$ forms a monotonic convergence sequence to $z$ as in Figure~\ref{non-Hausdorff_pt}. 
\begin{figure}[t]
\begin{center}
\includegraphics[scale=0.35]{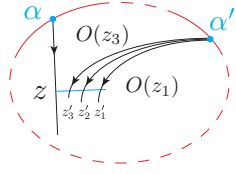}
\end{center}
\caption{A flow box $U$}
\label{non-Hausdorff_pt}
\end{figure}
Let $T_{a,b}$ be the open sub-arc of $T$ from a point $a \in T$ to a point $b \in T$. 
Denote by $\gamma_i$ the union of $\{ \alpha' \} \sqcup O(z_i) \sqcup O(z_{i+1}) \sqcup T(z_i', z_{i+1}')$.
Then $\gamma_i$ is a simple closed curve and let $B_i \subset U$ be the open disk in $U$ bounded by $\gamma_i$. 
The union $U_\infty := \bigsqcup_{i=1} B_i \sqcup O(z_i)$ is an open disk in $U$.
Since $\alpha$ and $\alpha'$ are distinct ends, the set difference $\partial_{v_{\mathrm{end}}} U_\infty - (\{ \alpha, \alpha' \} \sqcup O(z_1) \sqcup O(z))$ is not empty.  
The finiteness of singular points implies that $\partial_{v_{\mathrm{end}}} U_\infty - (\{ \alpha, \alpha' \} \sqcup O(z_1) \sqcup O(z) \sqcup T(z'_1, z))$ contains a non-singular point $w \in \overline{U} \subseteq S$. 
By construction, there is a convergence sequence $(w_i)_{i \in \Z_{>0}}$ to $w$ with $w_i \in O(z_i)$. 
This means that the pair of $z \in U \subseteq S - \mathop{\mathrm{BD}_+(v)}$ and $w$ is a non-Hausdorff pair. 
Since any non-Hausdorff points are contained in $\mathop{\mathrm{BD}_+(v)}$, we obtain $z \in \mathop{\mathrm{BD}_+(v)}$, contradicts $z \in S - \mathop{\mathrm{BD}_+(v)}$. 
By symmetry, we have that $\omega_{v_{\mathrm{end}}}(y) = \omega_{v_{\mathrm{end}}}(y')$ for any $y,y' \in U$.

Since $\mathop{\mathrm{BD}_+(v_{\mathrm{end}})}$ is a disjoint union of a finite graph and finitely many orbits periodic orbits, any connected components of $S - \mathop{\mathrm{BD}_+(v)} = S_{\mathrm{end}} - \mathop{\mathrm{BD}_+(v_{\mathrm{end}})}$ contribute at most finitely many hyperbolic or maximal elliptic or parabolic sectors. 
This means that any singular points in $S_{\mathrm{end}}$ are finitely sectored. 
Therefore the ends are sectored. 
\end{proof}

Theorem~\ref{main:01} and Lemma~\ref{lem:condition02} implies Theorem~\ref{thm:02}. 
%
%

\section{Examples}

We state the necessity of degeneracy and isolated properties of connected components of the singular point set. 

\subsection{Necessity of degeneracy of connected components of the singular point set}
Notice that the flow as in the left of Figure~\ref{fig:deg_sing01} is Hamiltonian such that every orbit is contained in a straight line $\R \times \{y \} \subset \R^2$ for some $y \in \R$. 
On the other hand, there is a non-Hamiltonian flow on the plane $\R^2$ whose singular point set is a closed interval such that every orbit is contained in a straight line $\R \times \{y \} \subset \R^2$ for some $y \in \R$. 
Indeed, let $X := (1,0)$ be a vector field and a smooth function $f \colon \R^2 \to [0,1]$ such that $f^{-1}(0) = \{0 \} \times [-1/2,1/2]$ and $\R^2 - (-3/4,3/4)^2 \subset f^{-1}(1)$. 
The flow $v$ generated by the vector field $fX = (f, 0)$ satisfies that the singular point set $\Sv$ is a closed interval such that every orbit is contained in a straight line $\R \times \{y \} \subset \R^2$ for some $y \in \R$ as in the middle of Figure~\ref{fig:deg_sing01}. 
\begin{figure}[t]
\begin{center}
\includegraphics[scale=0.4]{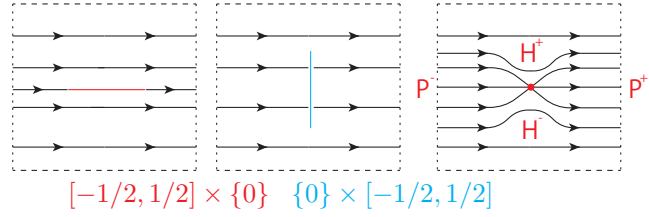}
\end{center}
\caption{A flow whose singular point set is a closed interval $\{0 \} \times [-1/2,1/2]$ and a flow with a singular point consisting of two parabolic sectors and two hyperbolic sectors.}
\label{fig:deg_sing01}
\end{figure}
We claim that $v$ is not Hamiltonian. 
Indeed, assume that $v$ is Hamiltonian. 
Then there is a Hamiltonian vector field $X_H$ whose generating flow is topologically equivalent to $v$ such that $X_H|_{[-1,1]^2 - (-3/4,3/4)^2} = (1,0)$. 
Let $H$ be the Hamiltonian of $X_H$. 
Consider a closed annulus $\A := [-1,1]/\sim \times  [-1,1]$, where the equivalence relation $\sim$ on $[-1,1]$ is defined by $x \sim y$ if either $x=y$ or $x = -y \in \{-1, 1\}$. 
Since $X_H|_{[-1,1]^2 - (-3/4,3/4)^2} = (1,0)$, the Hamiltonian $H$ induces the Hamiltonian $H_{\A}$ on $\A$. 
Therefore the Hamiltonian vector field$X_{H_{\A}}$ of $H_{\A}$ is also the induced vector field by $X_H$ and so has a wandering domain, which contradicts that Hamiltonian flows on compact surfaces are non-wandering. 

\subsection{Necessity of isolated properties of connected components of the singular point set}

Notice that the flow generated by a vector field $X =(1,0)$ is Hamiltonian such that every orbit is contained in a straight line $\R \times \{y \} \subset \R^2$ for some $y \in \R$. 
On the other hand, we have the following non-Hamiltonian flow on the plane whose singular point set is totally disconnected such that every orbit is contained in a straight line parallel to the $x$-axis. 

\begin{proposition}\label{prop:ex}
There is a non-Hamiltonian flow on the plane $\R^2$ generated by a vector field $fX$ with totally disconnected singular points, where $X =(1,0)$ is a Hamiltonian vector field and $f \colon \R^2 \to \R$ is a smooth function with totally disconnected zeros. 
In particular, every orbit is contained in a straight line $\R \times \{y \} \subset \R^2$ for some $y \in \R$. 
\end{proposition}

\begin{proof}
Replacing the restriction of flows on the closed square $[-3/4,3/4]^2$ in the previous example with a trivial flow box and applying \cite[
Theorem~D]{yokoyama2022omega} to the trivial flow box, we obtain a flow on the plane $\R^2$ with totally disconnected singular points each whose orbits are contained in $\R \times \{ y\}$ for some $y \in \R$. 
As the same argument in the previous subsection, the resulting flow is not Hamiltonian. 
\end{proof}

\subsection{Necessity of non-existence of parabolic sectors}

A flow with a singular point consisting of two parabolic sectors and two hyperbolic sectors, as in the right in Figure~\ref{fig:deg_sing01}, is not Hamiltonian, which follows from the following observation. 

\begin{lemma}
Any flows with parabolic sectors on surfaces are not Hamiltonian. 
\end{lemma}

\begin{proof}
Assume that there is a Hamiltonian flow $v$ with parabolic sectors on a surface $S$.
Let $H$ be a Hamiltonian whose generating flow $v_H$ is topologically equivalent to $v$ and $x \in S$ the singular point for the parabolic sector. 
Fix two distinct non-singular orbits $O, O'$ of $v_H$ contained in a parabolic sector. 
Then the values $H(O)$ and $H(O')$ are distinct. 
Since any inverse image $H^{-1}(r)$ of any value $r \in \R$ is closed, we obtain $x \in H^{-1}(H(O)) \cap H^{-1}(H(O')) = \emptyset$, which is a contradiction. 
\end{proof}

Notice that the resulting flow from the flow as in the right in Figure~\ref{fig:deg_sing01} by removing the unique singular point is a Hamiltonian flow without singular points on an open annulus. 

\vspace{10pt}

{\bf Acknowledgement}: 
The author would like to thank Prof. Mitsuhiro Shishikura for proposing the problem in the finite volume case.

\bibliographystyle{abbrv}
\bibliography{yt20211011,../../bib/yt,../../bib/yokoyama}

\end{document}